\documentclass[]{amsart}
\usepackage{mathrsfs}
\usepackage{color}
\usepackage{amsmath}
\usepackage{amsfonts}
\usepackage{amssymb}
\usepackage{graphicx}
\usepackage{hyperref}

 \newtheorem{Theorem}{Theorem}[section]
 \newtheorem{Corollary}[Theorem]{Corollary}
 \newtheorem{Lemma}[Theorem]{Lemma}
 
 \newtheorem{Question}[Theorem]{Question}

 \newtheorem{Conjecture}[Theorem]{Conjecture}

 \newtheorem{Remark}[Theorem]{Remark}

 \numberwithin{equation}{section}

\begin{document}

\title[A remark on a generalized Suita conjecture for finite points case]
 {A Remark on a generalized Suita conjecture for finite points case}

\author{Qi'an Guan}
\address{Qi'an Guan: School of Mathematical Sciences, Peking University, Beijing 100871, China.}
\email{guanqian@math.pku.edu.cn}

\author{Xun Sun}
\address{Xun Sun: School of Mathematical Sciences, Peking University, Beijing 100871, China.}
\email{sunxun@stu.pku.edu.cn}

\author{Zheng Yuan}
\address{Zheng Yuan: State Key Laboratory of Mathematical Sciences, Academy of Mathematics and Systems Science, Chinese Academy of Sciences, Beijing 100190, China.}
\email{yuanzheng@amss.ac.cn}

\thanks{}



\date{}

\dedicatory{}

\commby{}


\begin{abstract}
In this article, we use a class of harmonic functions (maybe multi-valued) to study the equality part in a weighted version of Suita conjecture for higher derivatives and finite points case, and we obtain some sufficient and necessary conditions for the equality part to hold when the harmonic part $u$ of the weight is trivial.
\end{abstract}

\maketitle

\section{introduction}
Let $\Omega$ be an open Riemann surface which admits a nontrivial Green function $G_{\Omega}$. Let $z_0$ be a point in $\Omega$, and let $w$ be a local coordinate on a neighborhood $V_{z_0}\Subset \Omega$ of $z_0$ satisfying that $w(z_0)=0$. 
Then the logarithmic capacity (see \cite{L-K}) of $z_0$ is locally defined by 
\[c_{\beta}(z_0):=\exp\lim_{z\rightarrow z_0}(G_{\Omega}(z,z_0)-\log|w(z)|),\]
and the analytic capacity (see \cite{L-K}) of $z_0$ is defined by
\[c_B (z_0):=\sup\{|f'(z_0)|:f\in \mathcal{O}(\Omega), |f|\le 1, f(z_0)=0\}.\]
Let $B_{\Omega}(z_0)$ be the Bergman kernel function on $\Omega$, which is defined by 
\[B_{\Omega}(z_0):=\frac{2}{\inf\{\int_{\Omega}|\widetilde{F}|^2:\widetilde{F}\in H^0(\Omega,\mathcal{O} (K_{\Omega}))\&(\widetilde{F}-dw)\in(\mathcal{I}(2G_{\Omega}(\cdot,z_0))\otimes \mathcal{O} (K_{\Omega}))_{z_0} \}},\]
where $|\widetilde{F}|^2:=\sqrt{-1}\widetilde{F}\wedge\bar{\widetilde{F}}$ for any holomorphic $(1,0)$ form $\widetilde{F}$ on $\Omega$. 

In \cite{L-K}, Sario and Oikawa posed an open question: \emph{find a relation between the magtitudes of the quantities $\sqrt{\pi B_{\Omega}(z_0)}$, $c_{\beta}(z_0)$ and $c_B(z_0)$.}
 
As $G_{\Omega}(z,z_0)=\sup_{v\in\Delta_0(z_0)}v(z)$ (see \cite{L-K}), where $\Delta_0(z_0)$ is the set of negative subharmonic functions $v$ on $\Omega$ satisfying that $v-\log|w|$ has a locally finite upper bound near $z_0$, we know that $c_{\beta}(z_0)\ge c_B(z_0)$. In \cite{suita}, Suita  proved that
 \begin{Theorem}[\cite{suita}]
  $\pi B_{\Omega}(z_0)\ge (c_{B}(z_0))^2$  holds for any $z_0\in\Omega$, and the equality holds if and only if $\Omega$ is conformally equivalent to the unit disc less a (possible) closed set of inner capacity zero. 	
 \end{Theorem}
 
For the relation between $B_{\Omega}(z_0)$ and $c_{\beta}(z_0)$, Suita \cite{suita} proved $\pi B_{\Omega}(z_0)>(c_{\beta}(z_0))^2$
when $\Omega$ is a doubly connected domain with no degenerate boundary component, and posed the following conjecture (so called \textbf{Suita Conjecture}).
\begin{Conjecture}[\cite{suita}]
Let $\Omega$ be any open Riemann surface which admits a nontrivial Green function $G_{\Omega}$.
 $\pi B_{\Omega}(z_0)\ge (c_{\beta}(z_0))^2$  holds for any $z_0\in\Omega$, and the equality holds if and only if $\Omega$ is conformally equivalent to the unit disc less a (possible) closed set of inner capacity zero. 
\end{Conjecture}
In fact, a closed set of inner capacity zero is a polar set (locally singularity set of a subharmonic function).

The inequality part of Suita conjecture for bounded planar domains was proved by Blocki \cite{Blocki}, and the inequality part of Suita conjecture for open Riemann surfaces 
was proved by Guan-Zhou \cite{G-ZhouL2_CR}. In \cite{guan-zhou13ap}, Guan-Zhou established an optimal $L^2$ extension theorem in a general setting, and as an application, they proved the equality part of Suita conjecture, which completed the proof of Suita conjecture.

There are many generalized Suita conjectures. In \cite{yamada}, Yamada posed a harmonic weight version of Suita conjecture (called extended Suita conjecture), which was proved by Guan-Zhou in \cite{guan-zhou13ap}. In \cite{guan-mi-peking}, Guan-Mi proved a subharmonic weight version of Suita conjecture by using a concavity property of minimal $L^2$ integrals (see \cite{guan_general concave,guan-mi-peking}) and the solution of extended Suita conjecture, and the case that weights may not be subharmonic was proved by Guan-Yuan \cite{GY-concavity}. In \cite{G-M-Y}, Guan-Mi-Yuan established a concavity property of minimal $L^2$ integrals with Lebesgue measurable gain on weakly pseudoconvex K\"ahler manifolds, and proved a weighted Suita conjecture for higher derivatives. After that, we \cite{G-S-Y} gave some results on the set of points for the holding of the equality in a weighted version of Suita conjecture for higher derivatives, which is highly related to the Dirichlet problem when $\Omega$ is a planar domain bounded by finite analytic closed curves.

Note that the above generalized Suita conjectures are corresponding to the optimal $L^2$ extension problem from a single point to open Riemann surfaces. Thus, it is a nature question that \textbf{can one prove a generalized Suita conjecture for analytic subsets case?} 

Without using the solution of (extended) Suita conjecture, Guan-Yuan \cite{G-Y2} provided a characterization of the holding of equality in optimal jets $L^2$ extension problem from arbitrary analytic subsets to open Riemann surfaces, which proved a weighted version of Suita conjecture for higher derivatives and analytic subsets case. 

In this article, we focus on  the equality part in the weighted version of Suita conjecture for higher derivatives and finite points case, and we further discuss the conditions under which the equality holds.

We recall some notations (see \cite{OF81}, see also \cite{guan-zhou13ap,G-Y2}). Let $p:\Delta\rightarrow\Omega$ be the universal covering from the unit disc $\Delta$ to $\Omega$. We call the holomorphic function $f$ on $\Delta$ a
multiplicative function, if there is a character $\chi$, which is the representation of the fundamental group of $\Omega$, such that $g^*f=\chi(g)f$, where $|\chi|=1$ and $g$ is an element of the fundamental group of $\Omega$.
Denote the set of such kinds of $f$ by $\mathcal{O}^{\chi}(\Omega)$.  

It is known that for any harmonic function $u$ on $\Omega$,
there exists a $\chi_{u}$ and a multiplicative function $f_u\in\mathcal{O}^{\chi_u}(\Omega)$,
such that $|f_u|=p^{*}\left(e^{u}\right)$.
If $u_1-u_2=\log|f|$, then $\chi_{u_1}=\chi_{u_2}$,
where $u_1$ and $u_2$ are harmonic functions on $\Omega$ and $f$ is a holomorphic function on $\Omega$.
Recall that for the Green function $G_{\Omega}(z,z_0)$,
there exist a $\chi_{z_0}$ and a multiplicative function $f_{z_0}\in\mathcal{O}^{\chi_{z_0}}(\Omega)$ such that $|f_{z_0}(z)|=p^{*}\left(e^{G_{\Omega}(z,z_0)}\right)$, and $\chi_{z_0}=1$ holds if and only if $\Omega$ is conformally equivalent to the unit disc less a (possible) closed set of inner capacity zero (see \cite{suita}).

Let 
$Z_{0}:=\{z_1,z_2,\ldots,z_m\}\subset\Omega$
be a subset of $\Omega$ satisfying that 
$z_j\neq z_k$ for $j\neq k$. Let $w_j$ be a local coordinate on a neighborhood $V_{z_j}\Subset\Omega$ of $z_j$ satisfying that $w_j(z_j)=0$ for $j\in\{1,2,\ldots,m\}$, where $V_{z_j}\cap V_{z_k}=\emptyset$ for any $z_j\neq z_k$.
Let $c_{\beta}(z_j)$ be the logarithmic capacity  on $\Omega$, which is locally defined by 
\[c_{\beta}(z_j):=\exp\lim_{z\rightarrow z_j}(G_{\Omega}(z,z_j)-\log|w_j(z)|).\]

Now we consider the weighted Bergman kernel function for higher derivatives and finite points. 
Let $k_j$ be a nonnegative integer for $j\in\{1,2,\ldots,m\}$, $a_j$ be a nonzero complex number for $j\in\{1,2,\ldots,m\}$.  
Denote by $\mathcal{Z}:=(z_1,\ldots,z_m)$, $\mathcal{K}:=(k_1,\ldots,k_m)$, $\mathcal{A}:=(a_1,\ldots,a_m)$. 
Let $f$ be a holomorphic $(1,0)$ form on $V_0:=\cup_{i=1}^m V_{z_i}$ satisfying that $f=a_jw_j^{k_j}dw_j$ on $V_{z_j}$. Denote by
\begin{equation}
	\begin{split}
		&B_{\Omega,\rho}^{\mathcal{K}}(\mathcal{Z},\mathcal{A}):=\\
		&\frac{2}{\inf\{\int_{\Omega}|\widetilde{F}|^2\rho :\widetilde{F}\in H^{0}(\Omega,\mathcal{O}(K_{\Omega}))\& (\widetilde{F}-f)\in (\mathcal{I}(2(k_j+1)G_{\Omega}(\cdot,z_j))\otimes\mathcal{O}(K_{\Omega}))_{z_j}\,\forall j\}},
	\end{split}
	\nonumber
\end{equation}
where $\rho$ is a nonnegative Lebesgue measurable function on $\Omega$ and $|\widetilde{F}|^2:=\sqrt{-1}\widetilde{F}\wedge\bar{\widetilde{F}}$ for any holomorphic $(1,0)$ form $\widetilde{F}$ on $\Omega$. 

Let $c(t)$ be a positive measurable function on $(0,+\infty)$ satisfying that $c(t)e^{-t}$ is decreasing on $(0,+\infty)$ and $\int_0^{+\infty}c(s)e^{-s}ds<+\infty$. 
Denote by
\[\rho=e^{-(2\sum_{j=1}^m(k_j+1-p_j)G_{\Omega}(\cdot,z_j)+2v)}c(-\sum_{j=1}^m2p_jG_{\Omega}(\cdot,z_j)),\]
where $p_j>0$ is a constant for any $j\in\{1,2,\ldots,m\}$ and $v$ is a subharmonic function on $\Omega$. Denote by $\alpha_j:=\sum_{1\le l\le m,l\not=j}(k_l+1)G_{\Omega}(z_j,z_l)+v(z_j)$. Let us recall the solution of a weighted version of Suita conjecture for higher derivatives and finite points case.

\begin{Theorem}[\cite{G-Y2}]
\label{thm6}
Suppose that $v(z_j)>-\infty$ for $1\le j\le m$. Then
\[1\le \left(\int_0^{+\infty}c(s)e^{-s}ds\right)\left(\sum_{j=1}^m\frac{\pi|a_j|^2e^{-2\alpha_j}}{p_jc_{\beta}(z_j)^{2(k_j+1)}}\right)B_{\Omega,\rho}^{\mathcal{K}}(\mathcal{Z},\mathcal{A}).\]
holds. Moreover, equality
\begin{equation}
\label{1}1= \left(\int_0^{+\infty}c(s)e^{-s}ds\right)\left(\sum_{j=1}^m\frac{\pi|a_j|^2e^{-2\alpha_j}}{p_jc_{\beta}(z_j)^{2(k_j+1)}}\right)B_{\Omega,\rho}^{\mathcal{K}}(\mathcal{Z},\mathcal{A})
\end{equation} 
holds if and only if the following statements hold:

$(1)$ $v=\log|g|+u$, where $g$ is a holomorphic function such that $g(z_j)\neq 0$ for each $j\in\{1,2,\ldots,m\}$ and $u$ is a harmonic function on $\Omega$;

$(2)$ $\chi_{-u}=\prod_{1\le j\le m}\chi_{z_j}^{k_j+1}$, where $\chi_{-u}$ and $\chi_{z_j}$ are the characters associated to the functions $-u$ and $G_{\Omega}(\cdot,z_j)$ respectively;

$(3)$ $\lim_{z\rightarrow z_k}\frac{f}{gp_*(f_u(\prod_{1\le j\le m}f_{z_j}^{k_j+1})(\sum_{1\le j\le m}p_j\frac{df_{z_j}}{f_{z_j}}))}=c_0$ for any $k\in\{1,2,\ldots,m\}$, where $c_0\in\mathbb{C}\backslash\{0\}$ is a constant independent of $k$.
\end{Theorem}

Note that the holomorphic $(1,0)$ form $p_*(f_u(\prod_{1\le j\le m}f_{z_j}^{k_j+1})(\sum_{1\le j\le m}p_j\frac{df_{z_j}}{f_{z_j}}))$ in statement $(3)$ is a single valued holomorphic $(1,0)$ form on $\Omega$ when statement $(2)$ holds.

\subsection{Main result}
In this article, we present some results about the equality part in the weighted version of Suita conjecture for higher derivatives and finite points case.

We continue to use the notations in Theorem \ref{thm6}. Throughout this paper, we consistently assume $v=\log|g|+u$, where $g\not\equiv0$ is a holomorphic function on $\Omega$ and $u$ is a harmonic function on $\Omega$. 

Let $\Pi_1$ be a subset of the fundamental group $\pi_1(\Omega)$ which generates $\pi_1(\Omega)$ (when $\Omega$ is simply connected, set $\Pi_1=\{0\}$). In this article, $\Pi_1$ generates $\pi_1(\Omega)$ means that there is no proper subgroup of $\pi_1(\Omega)$ containing $\Pi_1$.
Denote by $\widetilde{d}:=\frac{\partial-\bar{\partial}}{i}$. For every $\alpha\in \Pi_1$, $\gamma_{\alpha}$ denotes a piecewise smooth closed curve on $\Omega$, which represents $\alpha$.
Let $p:\Delta\rightarrow\Omega$ be the universal covering of $\Omega$.
By Lemma \ref{lemma7}, there exists a harmonic function $\widetilde u_{\alpha}$ on $\Delta$ such that
\begin{equation}\label{eq:1118b}
	\frac{1}{2\pi}\left(\widetilde u_{\alpha}(z')-\int_{\gamma_{\alpha}}\widetilde{d}G_{\Omega}(\cdot,p(z'))\right)\in\mathbb{Z}
\end{equation}
for any $z'\in \Delta\backslash{p^{-1}(\gamma_{\alpha})}$. 
Condition \eqref{eq:1118b} implies that $\widetilde{u}_{\alpha}$ is unique up to addition of a multiple of $2\pi$, and we will show that $\widetilde{u}_{\alpha}$ depends only on the homology class of $\gamma_{\alpha}$ in $\Omega$ (up to addition of a multiple of $2\pi$), which in particular
shows that $\widetilde{u}_{\alpha}$ depends only on $\alpha\in \pi_1(\Omega)$. It can be inferred from Condition \eqref{eq:1118b} that $\frac{1}{2\pi}(\beta^*\widetilde{u}_{\alpha}-\widetilde{u}_{\alpha})$ is an integer constant function for any $\beta\in\pi_1(\Omega)$.

Based on Theorem \ref{thm6}, we use $\{\widetilde{u}_{\alpha}\}_{\alpha\in\Pi_1}$ to give a characterization of the holding of the equality in the weighted version of Suita conjecture for higher derivatives and finite points case. 

\begin{Theorem}
\label{thm1}
Given $m$ distinct points $z_1,\ldots,z_m$ satisfying that $g(z_i)\neq 0$ for each $i$, and $m$ nonnegative integers $k_1,\ldots,k_m$, 
equality
\begin{equation}
\label{eq2}
1= \left(\int_0^{+\infty}c(s)e^{-s}ds\right)\left(\sum_{j=1}^m\frac{\pi|a_j|^2e^{-2\alpha_j}}{p_jc_{\beta}(z_j)^{2(k_j+1)}}\right)B_{\Omega,\rho}^{\mathcal{K}}(\mathcal{Z},\mathcal{A})
\end{equation}
holds for some numbers $a_1,\ldots,a_m$ if and only if
\begin{equation}
\label{eq1}\sum_{j=1}^m\left(\frac{k_j+1}{2\pi}\right)\widetilde{u}_{\alpha}(\widetilde{z_j})+\frac{1}{2\pi}\int_{\gamma_{\alpha}}\widetilde{d}u\in \mathbb{Z} 
\end{equation}
for any $\alpha\in\Pi_1$ and $\widetilde{z_j}\in p^{-1}(z_j)$ for $1\le j\le m$. 
\end{Theorem}

Following from Theorem \ref{thm6} and Theorem \ref{thm1}, we know that
if the condition \eqref{eq1} is satisfied, the equality \eqref{eq2} holds for  numbers $a_1,\ldots,a_m$ if and only if there exists a constant $c_0\in\mathbb{C}\backslash\{0\}$ such that
\[\lim_{z\rightarrow z_k}\frac{a_jw_j^{k_j}dw_j}{gp_*(f_u(\prod_{1\le j\le m}f_{z_j}^{k_j+1})(\sum_{1\le j\le m}p_j\frac{df_{z_j}}{f_{z_j}}))}=c_0\]
holds for any $k\in\{1,2,\ldots,m\}$.

When $m=1$, the above theorem can be referred to \cite{G-S-Y}.
\begin{Remark}
\label{remark3}
Suppose that $\Omega$ is conformally equivalent to the unit disc less a closed set of inner capacity zero. Then all $\frac{1}{2\pi}\widetilde{u}_{\alpha}$ are integer-valued functions on the unit disc $\Delta$. As a result, when $v=\log|g|$, condition \eqref{eq1} is always satisfied. 
On the other hand, if all $\frac{1}{2\pi}\widetilde{u}_{\alpha}$ are integer-valued functions on the unit disc $\Delta$, then $\Omega$ is conformally equivalent to the unit disc less a closed set of inner capacity zero. We will prove the present remark in Section \ref{section1}.
\end{Remark}

Note that for any $m$ points $z_1,\ldots,z_m$ and any $m$ nonnegative integers $k_1,\ldots,k_m$, there exists a harmonic function $u$ on $\Omega$ such that $\prod_{1\le j\le m}\chi_{z_j}^{k_j+1}=\chi_{-u}$ (see \cite{G-Y2}). Following from Theorem  \ref{thm6}, there exist $m$ numbers $a_1,\ldots,a_m$ such that equality \eqref{eq2} holds. Thus, it is a natural question that

\begin{Question}\label{q:1}
Can one obtain some sufficient or necessary conditions for equality \eqref{eq2} to hold when $u=0$?
\end{Question}

Recall that when $u=0$, $v=\log |g|$, where $g\not\equiv 0$ is a holomorphic function on $\Omega$. Based on Theorem \ref{thm6}, the above question focuses on the statement $\prod_{1\le j\le m}\chi_{z_j}^{k_j+1}=1$. 

From Remark \ref{remark3}, we can see that if $\Omega$ is not conformally equivalent to the unit disc less a (possible) closed set of inner capacity zero, there exists an element $\alpha\in\Pi_1$ such that $\frac{1}{2\pi}\widetilde{u}_{\alpha}$ is not an integer-valued function on the unit disc $\Delta$. 
As a result, condition \eqref{eq1} is not always satisfied. 
The following corollary gives a sufficient condition for equality \eqref{eq2} to hold when $u=0$ and $\{\frac{1}{2\pi}\widetilde{u}_{\alpha}\}$ are not all constant integer-valued functions.
\begin{Corollary}
\label{coro2}
Suppose that $u=0$ and $\frac{1}{2\pi}\widetilde{u}_{\alpha_1},\ldots,\frac{1}{2\pi}\widetilde{u}_{\alpha_m}$ are not constant functions on the unit disc, while $\frac{1}{2\pi}\widetilde{u}_{\alpha}$ are constant integer-valued functions on the unit disc $\Delta$ for other $\alpha\in\Pi_1$. 
Fix $m$ distinct points $\{z'_1,\ldots,z'_m\}\subset\Delta\backslash p^{-1}(g^{-1}(0))$. If the matrix 
\[\left(\frac{\partial\widetilde{u}_{\alpha_k}}{\partial z}(z_j')\right)_{1\le k\le m,1\le j\le m}\] 
is nondegenerate, then there exist $m$ different points $\{{z}_1,\ldots,{z}_m\}\subset\Omega\backslash g^{-1}(0)$, $m$ nonnegative integers $k_1,\ldots,k_m$ and $m$ numbers $a_1,\ldots,a_m$, such that equality  
\[1= \left(\int_0^{+\infty}c(s)e^{-s}ds\right)\left(\sum_{j=1}^m\frac{\pi|a_j|^2e^{-2\alpha_j}}{p_jc_{\beta}(z_j)^{2(k_j+1)}}\right)B_{\Omega,\rho}^{\mathcal{K}}(\mathcal{Z},\mathcal{A})\]
holds.
\end{Corollary}

\

In the following part, we assume that $\Omega\subset\mathbb{C}$ is a bounded  planar domain which is bounded by $n$ analytic Jordan curves and we do not require that $u=0$.
Denote by $\mathbb{C}\backslash \Omega:=\bigcup_{i=1}^{n}A_k$, where $A_k$ are disjoint closed sets, and $A_n$ is the unbounded  connected component of $\mathbb{C}\backslash\Omega$. 
Denote the boundary of $A_k$ by $\Gamma_k$, $1\le k\le n$.
For any $1\le k\le n-1$, let $\gamma_k$ be a closed smooth curve in $\Omega$, which winds around points in $A_k$ once and does not wind around points in $A_k$ for $l\neq k$ (for details, see \cite{fulton,G-S-Y}).

For $1\le k\le n-1$, by solving a Dirichlet problem, there exists $u_k\in C(\overline{\Omega})$, which satisfies that $u_k=1$ on $\Gamma_k$, $u_k=0$ on $\Gamma_l$ for $l\neq k$ and $u_k$ is harmonic on $\Omega$.

Note that the homotopy classes of $\gamma_1,\ldots,\gamma_{n-1}$ generate $\pi_1(\Omega)$. We can therefore choose $\Pi_1$ in Theorem \ref{thm1} to be $\{[\gamma_1],\ldots,[\gamma_{n-1}]\}$, where $[\gamma_i]$ represents the homotopy class of $\gamma_i$ for $1\le i\le n-1$. Then, using Theorem \ref{thm1}, we obtain the following result.

\begin{Theorem}
\label{thm2}
Given $m$ distinct points $z_1,\ldots,z_m$ satisfying that $g(z_i)\neq 0$ for each $i$, and $m$ nonnegative integers $k_1,\ldots,k_m$, 
equality 
\begin{equation}
\label{eq5}
1=\left(\int_0^{+\infty}c(s)e^{-s}ds\right)\left(\sum_{j=1}^m\frac{\pi|a_j|^2e^{-2\alpha_j}}{p_jc_{\beta}(z_j)^{2(k_j+1)}}\right)B_{\Omega,\rho}^{\mathcal{K}}(\mathcal{Z},\mathcal{A})
\end{equation}
holds for some numbers $a_1,\ldots,a_m$ if and only if
\begin{equation}
\label{eq4}
\sum_{1\le j\le m}(k_j+1)u_k(z_j)+\frac{1}{2\pi}\int_{\gamma_k}\frac{\partial u(z)}{\partial n}ds(z)\in \mathbb{Z}
\end{equation} 
holds for any $1\le k\le n-1$, where $\frac{\partial}{\partial n}$ means the differentiation in the direction of the outward pointing normal, and $s$ is the arc-length parameter.  
\end{Theorem}

\begin{Remark}
\label{remark1}
Recall that for each $[\gamma_k]\in \Pi_1$, we have constructed a harmonic function $\widetilde{u}_k$ on the unit disc $\Delta$ satisfying the condition \eqref{eq:1118b}.

Since $\Omega$ is a planar domain, $p_*\widetilde{u}_{k}$ is a single-valued function on $\Omega$ (see \cite{G-S-Y}) for each $1\le k\le n-1$ and is harmonic on $\Omega$, and $p_*\widetilde{u}_{k}$ equals to $2\pi u_k$ plus a multiple
of $2\pi$ for each $1\le k\le n-1$.
We prove the present remark in Section \ref{section2}.
\end{Remark}

The following remark gives an example of Theorem \ref{thm2}.
\begin{Remark}
\label{remark2}
 Suppose that $\Omega=\{z\in\mathbb{C}:1<|z|<R\}$. Denote that
\[c:=\int_{ \{|z|=\frac{R+1}{2}\}}\frac{\partial u}{\partial n}ds(z)\]
is a constant. Note that $u_1=\frac{\log R-\log|z|}{\log R}$ is a harmonic function on $\Omega$ such that $u_1(z)=1$ when $|z|=1$ and $u_1(z)=0$ when $|z|=R$. 
Using Theorem \ref{thm2}, we can see that given $m$ nonnegative integers $k_1,\ldots,k_m$ and $m$ distinct points $z_1,\ldots,z_m$,
the equality
\[1= \left(\int_0^{+\infty}c(s)e^{-s}ds\right)\left(\sum_{j=1}^m\frac{\pi|a_j|^2e^{-2\alpha_j}}{p_jc_{\beta}(z_j)^{2(k_j+1)}}\right)B_{\Omega,\rho}^{\mathcal{K}}(\mathcal{Z},\mathcal{A})\]
holds for some numbers $a_1,\ldots,a_m$ if and only if
\[|z_1|^{k_1+1}\cdots|z_m|^{k_m+1}=R^{k_1+\cdots+k_m+m+c-N},\] 
where $N$ is an integer.
\end{Remark}

In the following, we always assume that $u=0$ and give some answers to Question \ref{q:1}.

The following corollary shows that equality \eqref{eq:1120} can not hold for any $m$ different points $z_1,\ldots,z_m$ and any $m$ numbers $a_1,\ldots,a_m$ when the connectivity $n$ of $\Omega$  is sufficiently large and $u=0$.

\begin{Corollary}\label{c:n}
Let $k_1,\ldots,k_m$ be arbitrary $m$ nonnegative integers. Then
equality 
\begin{equation}
\label{eq:1120}
	1=\left(\int_0^{+\infty}c(s)e^{-s}ds\right)\left(\sum_{j=1}^m\frac{\pi|a_j|^2e^{-2\alpha_j}}{p_jc_{\beta}(z_j)^{2(k_j+1)}}\right)B_{\Omega,\rho}^{\mathcal{K}}(\mathcal{Z},\mathcal{A})
\end{equation}
does not hold for any $m$ distinct points $z_1,\ldots,z_m$ and any $m$ numbers $a_1,\ldots,a_m$ unless 
$$n\le\sum_{1\le j\le m}(k_j+1).$$
\end{Corollary}
When $m=1$, the above Corollary can be seen in \cite{G-S-Y}, and when $m=1$ and $k_1=0$,  the above Corollary is  the solution of the equality part of Suita conjecture (see \cite{guan-zhou13ap}) for the case that $\Omega$ is a   planar domain  bounded by  finite analytic Jordan curves.

Take arbitrary $m$ distinct points $z_1,\ldots,z_m$ and arbitrary $m$ numbers $a_1,\ldots,a_m$. The optimal jets $L^2$ extension theorem (see \cite{G-Y2,guan-zhou13ap}) shows that  there exists a holomorphic $(1,0)$ form $F$ on $\Omega$, such that $F=a_j(z-z_j)^{k_j}dz+o((z-z_j)^{k_j})dz$ near $z_j$ for any $j$ and 
\begin{equation}
	\label{eq:241119}
	\int_{\Omega}|F|^2\rho\le \left(\int_0^{+\infty}c(s)e^{-s}ds\right)\left(\sum_{j=1}^m\frac{\pi|a_j|^2e^{-2\alpha_j}}{p_jc_{\beta}(z_j)^{2(k_j+1)}}\right).
	\end{equation}
Corollary \ref{c:n} indicates that 
if $n> \sum_{1\le j\le m}(k_j+1)$, the ``$\le$" in inequality \eqref{eq:241119} can be refined  to ``$<$".

The following theorem gives the existence of $(\mathcal{K},\mathcal{Z},\mathcal{A})$ such that equality \eqref{1} holds when $m\ge n-1$.

\begin{Theorem}
\label{thm9}
Assume that $n>1$, where  $n$ is the connectivity of $\Omega$. For any integer $m\ge n-1$, there exist $m$ distinct points $z_1,\ldots,z_{m}$, nonnegative integers $k_1,\ldots,k_{m}$ and $m$ numbers $a_1,\ldots,a_{m}$ such that  equality 
\[1= \left(\int_0^{+\infty}c(s)e^{-s}ds\right)\left(\sum_{j=1}^m\frac{\pi|a_j|^2e^{-2\alpha_j}}{p_j c_{\beta}(z_j)^{2(k_j+1)}}\right)B_{\Omega,\rho}^{\mathcal{K}}(\mathcal{Z},\mathcal{A})\]
holds.
\end{Theorem}

In \cite{G-S-Y}, we proved that if $\Omega\subset\mathbb{C}$  is bounded by three analytic Jordan curves and $v\equiv 0$, then there exists a point $z_1\in\Omega$ 
and a large enough integer $k_1$ such that for any $a_1\neq 0$,
\[1= \left(\int_0^{+\infty}c(s)e^{-s}ds\right)\left(\frac{\pi|a_1|^2}{p_1c_{\beta}(z_1)^{2(k_1+1)}} \right)B_{\Omega,\rho}^{k_1}(z_1;a_1).\]
We  generalize this result to arbitrary $n\ge 3$, where $n$ is the connectivity of $\Omega$.

\begin{Theorem}
\label{thm10}
Assume that $n>2$. Then there exist $n-2$ distinct points $z_1,\ldots,z_{n-2}$, nonnegative integers $k_1,\ldots,k_{n-2}$, 
and numbers $a_1,\ldots,a_{n-2}$ such that equality 
\[1= \left(\int_0^{+\infty}c(s)e^{-s}ds\right)\left(\sum_{j=1}^{n-2}\frac{\pi|a_j|^2e^{-2\alpha_j}}{p_jc_{\beta}(z_j)^{2(k_j+1)}}\right)B_{\Omega,\rho}^{\mathcal{K}}(\mathcal{Z},\mathcal{A})\]
holds, where $\mathcal{Z}$ denotes $(z_1,\ldots,z_{n-2})$, $\mathcal{K}$ denotes $(k_1,\ldots,k_{n-2})$, and $\mathcal{A}$ denotes $(a_1,\ldots,a_{n-2})$.
\end{Theorem}

In \cite{X-Z}, using Theorem 1.9 in \cite{G-S-Y} (see also Theorem \ref{thm2} for the case $m=1$), Xu-Zhou proved that when $m=1$ and $n\ge 3$, given any $M\gg1$, there exists a family of
smoothly bounded $n$-connected domain $\Omega$ such that no point of $\Omega$ can satisfy the
equality in $k$-order Suita conjecture, where $k= 0,1,\ldots,M$. We can generalize this result  to arbitrary $m\ge 1$ using Theorem \ref{thm2}.
 
\begin{Theorem}
\label{thm11}
Let $m$ be an arbitrary positive integer, $n\ge m+2$. Given any $M\gg1$, there exists a family of
smoothly bounded $n$-connected domain $\Omega$ such that when $u=0$ on $\Omega$, $z_1,\ldots,z_m$ are arbitrary distinct points of $\Omega$, $k_1,\ldots,k_m$ are arbitrary nonnegative integers such that $k_i\le M$ for each $i$, $a_1,\ldots,a_m$ are arbitrary numbers, equality
\[1= \left(\int_0^{+\infty}c(s)e^{-s}ds\right)\left(\sum_{j=1}^m\frac{\pi|a_j|^2e^{-2\alpha_j}}{p_j c_{\beta}(z_j)^{2(k_j+1)}}\right)B_{\Omega,\rho}^{\mathcal{K}}(\mathcal{Z},\mathcal{A})\]
does not hold.
\end{Theorem}

\section{preparation}
In this section, we will do some preparations.

\subsection{Characters on open Riemann surfaces}\quad\par
Let $\Omega$ be an open Riemann surface admitted a nontrivial Green function $G_{\Omega}$. Let $p:\Delta\rightarrow\Omega$ be the universal covering from the unit disc $\Delta$
to $\Omega$. Let $z_0$ and $z_1$ be different points in $\Omega$. Since $p:\Delta\rightarrow\Omega$ is the universal covering, any element in $\pi_1(\Omega,z_0)$ can be represent by $p\circ\widetilde{\gamma}$, where $\widetilde{\gamma}$ is a curve in $\Delta$, 
and it depends only on the starting point and end point of $\widetilde{\gamma}$. Therefore we can assume that $\widetilde{\gamma}$ is piecewise smooth and lies in $\Delta\backslash p^{-1}(z_1)$. Thus, any element $\alpha\in\pi_1(\Omega,z_0)$ can be represented by a 
piecewise smooth curve $\gamma_{\alpha}\subset\Omega\backslash\{z_1\}$.

Recall that
$\widetilde{d}=\frac{\partial-\bar{\partial}}{i},$
and $u$ is a harmonic function on $\Omega$.
\begin{Lemma}
[see \cite{G-S-Y}]
\label{lemma1}    
Let $\alpha$ be an element in $\pi_1(\Omega,z_0)$, $\gamma_{\alpha}\subset\Omega\backslash\{z_1\}$ be a piecewise smooth curve representing $\alpha$. Then we have
\[\chi_{z_1}(\alpha)=e^{i\int_{\gamma_{\alpha}}\widetilde{d}G_{\Omega}(\cdot,z_1)}\]
and
\[\chi_{u}(\alpha)=e^{i\int_{\gamma_{\alpha}}\widetilde{d}u}.\]
\end{Lemma}

Let $\gamma\subset\Omega$ be a piecewise smooth closed curve.  
The following lemma shows that the pull-back of the function $\int_{\gamma}\widetilde{d}G_{\Omega}(\cdot,z)$ defined in $\Omega\backslash\gamma$ can be extended to a harmonic function on $\Delta$ modulo integer multiples of $2\pi$.
\begin{Lemma}
[see \cite{G-S-Y}]
\label{lemma7}
There exists a harmonic function $H(z')$ on $\Delta$ such that $H(z')=\int_{\gamma}\widetilde{d}G_{\Omega}(\cdot,p(z'))+2k\pi$ for $z'\in\Delta\backslash p^{-1}(\gamma)$, where $k$ is an integer depending on $z'$.
\end{Lemma}

Recall that $\Pi_1$ is a subset of $\pi_1(\Omega)$ which can generate $\pi_1(\Omega)$. The following lemma will be used in the proof of Theorem \ref{thm1}.

\begin{Lemma} 
\label{lemma2}
$\chi_{-u}=\prod_{1\le j\le m}\chi_{z_j}^{k_j+1}$ holds if and only if 
\[e^{-i\int_{\gamma_{\alpha}}\widetilde{d}u}=e^{i\sum_{1\le j\le m}(k_j+1)\int_{\gamma_{\alpha}}\widetilde{d}G_{\Omega}(\cdot,z_j)}\] 
holds for any $\alpha\in\Pi_1$ and piecewise smooth curve $\gamma_{\alpha}\subset\Omega\backslash\{z_1,\ldots,z_m\}$ which can represent $\alpha$.
\end{Lemma}
	
\begin{proof}
It follows from Lemma \ref{lemma1} that 
\[\chi_{-u}=\prod_{1\le j\le m}\chi_{z_j}^{k_j+1}\]
holds if and only if 
\[e^{-i\int_{\gamma}\widetilde{d}u}=e^{i\sum_{1\le j\le m}(k_j+1)\int_{\gamma}\widetilde{d}G_{\Omega}(\cdot,z_j)}\] 
holds for any smooth closed curve $\gamma\subset\Omega\backslash\{z_1,\ldots,z_m\}$ with starting point and end point $z_0$. 

For $1\le l\le m$, let $\gamma_l'$ be a smooth closed curve in $U_{l}\backslash\{z_l\}\subset\Omega\backslash\{z_1,\ldots,z_m\}$, where $U_{l}$ is a small enough local coordinate ball with the center $z_{l}$ such that $z_{l'}\not\in U_l$ for $l'\not=l$. We require that $\gamma_l'$ winds around $z_l$ once. 
Since the first homology group $H_1(\Omega)$ is the abelization of $\pi_1(\Omega)$, it is generated by the homology class of element of $\Pi_1$. 
Note that $H_1(\Omega\backslash\{z_1,\ldots,z_m\})\cong H_1(\Omega)\oplus\mathbb{Z}^{ m}$, 
and the homology classes of element of $\Pi_1$ and the homology classes of $\gamma_l'$ for $1\le l\le m$ generate $H_1(\Omega\backslash\{z_1,\ldots,z_m\})$. 
	
Let $\gamma\subset\Omega\backslash \{z_1,\ldots,z_m\}$ be a smooth curve with starting point and end point $z_0$. There exists a sequence of integers $t_{\alpha}$ for $\alpha\in\Pi_1$ (finite many $t_{\alpha}$ are nonzero) and $t_l'$ for $1\le l\le m$ such that 
$\sum_{\alpha\in\Pi_1}t_{\alpha} \gamma_{\alpha}+\sum_{l=1}^m t_{l}'\gamma_l'$ is homologous to $\gamma$ in $\Omega\backslash\{z_1,\ldots,z_m\}$, where $\gamma_{\alpha}\subset\Omega\backslash\{z_1,\ldots,z_m\}$ is a piecewise smooth curve representing $\alpha$.	
	
For $1\le j\le m$, $G_{\Omega}(z,z_j)$ and $u$ are harmonic functions on $\Omega\backslash\{z_1,\ldots,z_m\}$, $\widetilde{d} G_{\Omega}(\cdot,z_j)$ and $\widetilde{d} u$ are therefore closed differential 1-forms on $\Omega\backslash\{z_1,\ldots,z_m\}$. 
Therefore we have   
\[\int_{\gamma}\widetilde{d}G_{\Omega}(\cdot,z_j)=\sum_{\alpha\in\Pi_1}t_{\alpha}\int_{\gamma_k}\widetilde{d}G_{\Omega}(\cdot,z_j)+\sum_{l=1}^{m}t_l'\int_{\gamma_l'}\widetilde{d}G_{\Omega}(\cdot,z_j)\]
for $1\le j\le m$ and
\[\int_{\gamma}\widetilde{d}u=\sum_{\alpha\in\Pi_1}t_{\alpha}\int_{\gamma_{\alpha}}\widetilde{d}u+\sum_{l=1}^{m}t_l'\int_{\gamma_l'}\widetilde{d}u.\]
	
Note that for $1\le j\le m$, $G_{\Omega}(z,z_j)=\log|z-z_j|+u'_j(z)$ on $\Omega$ on a local coordinate neighborhood $U_{j}$ of $z_j$ with coordinate $z$, where $u'_j$ is a harmonic function on $U_{j}$, hence 
\[\int_{\gamma'_l}\widetilde{d}G_{\Omega}(\cdot,z_j)=\int_{\gamma_l'} \widetilde{d}\log|\cdot-z_j|+\int_{\gamma_l'} \widetilde{d}u_j'=0\]
for $1\le l\le m$ and $l\neq j$, 
\[\int_{\gamma'_j}\widetilde{d}G_{\Omega}(\cdot,z_j)=\int_{\gamma_j'}\widetilde{d} \log|\cdot-z_j|+\int_{\gamma_j'}\widetilde{d} u_j'=2\pi\]
and
\[\int_{\gamma_l'}\widetilde{d}u=0.\]
for $1\le l\le m$.
	
From these results, we can see that 
\[e^{-i\int_{\gamma}\widetilde{d}u}=e^{i\sum_{1\le j\le m}(k_j+1)\int_{\gamma}\widetilde{d}G_{\Omega}(\cdot,z_j)}\] 
holds for any smooth closed curve $\gamma\subset\Omega\backslash\{z_1,\ldots,z_m\}$ with starting point and end point $z_0$ if and only if 
\[e^{-i\int_{\gamma_{\alpha}}\widetilde{d}u}=e^{i\sum_{1\le j\le m}(k_j+1)\int_{\gamma_{\alpha}}\widetilde{d}G_{\Omega}(\cdot,z_j)}\] 
holds for any $\alpha\in\Pi_1$ and piecewise smooth curve $\gamma_{\alpha}\subset\Omega\backslash\{z_1,\ldots,z_m\}$ representing $\alpha$.
Lemma \ref{lemma2} is therefore proved.
\end{proof}

\subsection{Some lemmas on planar domains}\quad\par
In this section, assume that $\Omega\subset\mathbb{C}$ is a bounded  planar domain  bounded by $n$ analytic Jordan curves and $\mathbb{C}\backslash\Omega=\bigcup_{j=1}^{n} A_j$, where $\{A_j\}$ are disjoint connected closed sets and $A_n$ is
the unbounded connected component of $\mathbb{C}\backslash\Omega$. Denote by $\Gamma_k$ the boundary of $A_k$, and  $\Gamma:=\cup_k\Gamma_k$.

Firstly, we recall the following result in algebraic topology.

\begin{Lemma}[see \cite{fulton}]\label{lemma8} 
Let $U$ be a bounded connected open subset of $\mathbb{C}$ such that $\mathbb{C}\backslash U=\bigcup_{j=1}^n A_j$, where $A_j$ are disjoint connected closed sets and $A_n$ is unbounded. For each $1\le j\le n-1$, there exists a closed smooth curve $\gamma_j$ in $\Omega$ which winds around points in $A_j$
once and does not wind around points in $A_l$ for $l\neq j$. If $\gamma_j'$ is an another closed smooth curve which satisfies these conditions, then $\gamma_j'$ is homologous to $\gamma_j$. 

In further, the homology classes of these $n-1$ curves form a basis for the first homology group $H_1(U)$, i.e., these curves are not homologous to each other and any curve in $U$ is homologous to some linear combinations of these curves. 
\end{Lemma}

Let us  recall  a famous result on the Dirichlet problem. 

\begin{Lemma}[see \cite{Ahlfors74} or \cite{G-S-Y}]
\label{lemma5}
The Dirichlet problem on $\Omega$ is always solvable, i.e., for any $f\in C(\Gamma)$, there is $\tilde f\in C(\overline{\Omega})$ s.t. $\tilde{f}|_{\Gamma}=f$  and $\tilde{f}$ is harmonic on $\Omega$.
\end{Lemma}

We will also need the following well-known formula.
\begin{Lemma}[The Green third formula, see \cite{nehari}]\label{lemma10}
	 Let $f\in C(\overline{\Omega})$, which is  harmonic on $\Omega$.  Then we have
	\[f(z_0)=\frac{1}{2\pi}\int_{\Gamma}f\frac{\partial G_{\Omega}(\cdot,z_0)}{\partial n}ds\]
	for $z_0\in\Omega$, where $\frac{\partial}{\partial n}$ means the differentiation in the direction of the outward pointing normal, 
	$s$ is the arc-length parameter.
\end{Lemma}

By Lemma \ref{lemma5}, for each $1\le k\le n-1$, we can get a unique function $u_k\in C(\overline{\Omega})$ such that $u_k=1$ on $\Gamma_k$, $u_k=0$ on $\Gamma_l$ for any $l\neq k$ and $u_k$ is harmonic on $\Omega$.

Let $m\ge n-1$, and let $\{\widetilde{z}_1,\ldots,\widetilde{z}_m\}$ be $m$ distinct points in $\Omega$. For $1\le j\le m$, take a simply connected neighborhood $V_j$ of $\widetilde{z}_j$ such that 
$V_i\cap V_j=\emptyset$ when $i\neq j$.
Every $u_k$ can be realized as the real part of a holomorphic function $f_k$ on $V_0:=\cup_{j=1}^m V_j$. 

The following lemma will be used in the proof of Theorems \ref{thm9} and \ref{thm10}.
\begin{Lemma}
\label{lemma9}
For each $1\le l\le n-1$, there exists a point $(z_1',\ldots,z_{l}')\in V_1\times\cdots \times V_{l}$ such that the matrix $\left(\frac{\partial f_k}{\partial z}(z_j')\right)_{1\le k\le l,1\le j\le l}$ has rank $l$.
In particular, there exists a point $(z_1',\ldots,z_{n-1}')\in V_1\times\cdots\times V_{n-1}$ such that the matrix $\left(\frac{\partial f_k}{\partial z}(z_j')\right)_{1\le k\le n-1,1\le j\le n-1}$ has rank $n-1$.

\end{Lemma}

\begin{proof}
We  prove this lemma by induction on $l$.

When $l=1$, since $f_1$ is not constant function on $V_1$, $\frac{\partial f_1}{\partial z}(z_1')\neq 0$ at some point $z_1'\in V_1$.

Assume that this lemma has been proved for $l-1$. By induction, we can assume that 
\[\det\left(\left(\frac{\partial f_k}{\partial z}(z_j')\right)_{1\le k\le l-1,1\le j\le l-1}\right)\neq 0\]
for some point $z_j'\in V_j$ for $1\le j\le l-1$.
Hence there exists unique $l-1$ numbers $k_1,\ldots,k_{l-1}$ such that
\begin{align} \nonumber
\left(\frac{\partial f_l}{\partial z}(z_1'),\ldots,\frac{\partial f_l}{\partial z}(z_{l-1}')\right)=\sum_{1\le j\le l-1}k_j\left(\frac{\partial f_{j}}{\partial z}(z_1'),\ldots,\frac{\partial f_{j}}{\partial z}(z_{l-1}')\right).
\end{align}
For any point $z_{l}'\in V_l$, it is easy to see that the inequality
\[\det\left(\left(\frac{\partial f_k}{\partial z}(z_j')\right)_{1\le k\le l,1\le j\le l} \right)\neq 0\]
holds if and only if
\[\frac{\partial f_{l}}{\partial z}(z_{l}')\neq k_1\frac{\partial f_1}{\partial z}(z_{l}')+\cdots+k_{l-1}\frac{\partial f_{l-1}}{\partial z}(z_{l}').\]
We can now complete the induction by contradiction. Suppose that
\[\frac{\partial f_{l}}{\partial z}=k_1\frac{\partial f_1}{\partial z}+\cdots+k_{l-1}\frac{\partial f_{l-1}}{\partial z}\]
holds for any point $z_{l}\in V_l$, then
\[\frac{\partial}{\partial z}(f_l-k_1f_1-\cdots-k_{l-1}f_{l-1})=0\]
on $V_l$, which implies that $f_l-k_1f_1-\cdots-k_{l-1}f_{l-1}$ is a constant function on $V_l$. $u_{l}-k_1u_1-\cdots-k_{n-2}u_{l-1}$ is therefore a constant function on $V_l$.
Note that $u_{l}-k_1u_1-\cdots-k_{n-2}u_{l-1}$ is a harmonic function on $\Omega$. Then, it is  a constant function on $\Omega$, which is a contradiction since it takes value 1 on $\Gamma_{l}$ and $0$ on $\Gamma_n$. 
Therefore there exists a point $z_{l}'\in V_{l}$ such that
$\frac{\partial f_l}{\partial z}(z_l')\neq k_1\frac{\partial f_1}{\partial z}(z_l')+\cdots+k_{l-1}\frac{\partial f_{l-1}}{\partial z}(z_{l}'),$
which implies that
$\det \left(\left(\frac{\partial f_k}{\partial z}(z_j')\right)_{1\le k\le l,1\le j\le l}\right)\neq 0.$
\end{proof}

\section{Proofs of Theorem \ref{thm1}, Remark \ref{remark3} and Corollary \ref{coro2} }
\label{section1}
In this section, we prove Theorem \ref{thm1}, Remark \ref{remark3} and Corollary \ref{coro2}.

\subsection{Proof of Theorem \ref{thm1}.}\quad\par

Using Lemma \ref{lemma7}, there exists a harmonic function $\widetilde u_{\alpha}$ on $\Delta$ such that
\begin{equation}\label{3.1}
	\frac{1}{2\pi}\left(\widetilde u_{\alpha}(z')-\int_{\gamma_{\alpha}}\widetilde{d}G_{\Omega}(\cdot,p(z'))\right)\in\mathbb{Z}
\end{equation}
for any $z'\in \Delta\backslash{p^{-1}(\gamma_{\alpha})}$.
Lemma \ref{lemma1} shows that $\frac{1}{2\pi}\widetilde u_{\alpha}(z')$ depends only on the homology class of $\gamma_{\alpha}$ (up to addition of an integer).

By lemma \ref{lemma2}, we know that
\[\chi_{-u}=\prod_{1\le j\le m}\chi_{z_j}^{k_j+1}\]
if and only if
\begin{align}
\label{align4}
e^{-i\int_{\gamma_{\alpha}}\widetilde{d}u}=e^{i\sum_{1\le j\le m}(k_j+1)\int_{\gamma_{\alpha}}\widetilde{d}G_{\Omega}(\cdot,z_j)}
\end{align}
holds for any $\alpha\in\Pi_1$ and piecewise smooth curves $\gamma_{\alpha}\subset\Omega\backslash\{z_1,\ldots,z_m\}$ representing $\alpha$.
Using condition \eqref{3.1}, 
we know that equality \eqref{align4} is equivalent to
\[\sum_{j=1}^m\left(\frac{k_j+1}{2\pi}\right)\widetilde{u}_{\alpha}(\widetilde{z_j})+\frac{1}{2\pi}\int_{\gamma_{\alpha}}\widetilde{d}u\in \mathbb{Z}\]
for $\widetilde z_j\in p^{-1}(z_j)$, $\alpha\in\Pi_1$ and piecewise smooth closed curves $\gamma_{\alpha}$ representing $\alpha$.
Thus, following from Theorem \ref{thm6}, 
Theorem \ref{thm1}  holds.

\subsection{Proof of Remark \ref{remark3}.}\quad\par
Suppose that $\Omega$ is the unit disc $\Delta$ less a closed set of inner capacity zero and $z_0\in\Omega$ is a fixed point. Since $\Omega\subset\Delta$, we know that $G_{\Omega}(z,z_0)\ge G_{\Delta}(z,z_0)$ for $z\in\Omega\backslash\{z_0\}$.
On the other hand, a closed set of inner capacity zero is locally the pole set of a subharmonic function, and $G_{\Omega}(z,z_0)$ is a negative subharmonic function on $\Omega$,
hence $G_{\Omega}(z,z_0)$ can be extended to a negative subharmonic function on $\Delta$ (see \cite{demailly1}), which implies that $G_{\Omega}(z,z_0)\le G_{\Delta}(z,z_0)$ for $z\in\Omega\backslash\{z_0\}$. Therefore, $G_{\Omega}(z,z_0)=G_{\Delta}(z,z_0)$ for any $z\in\Omega\backslash\{z_0\}$.

Given any piecewise smooth closed curve $\gamma\subset\Omega\backslash\{z_0\}$, it is clear that
\[\int_{\gamma}\widetilde{d}G_{\Omega}(\cdot,z_0)=\int_{\gamma}\widetilde{d}G_{\Delta}(\cdot,z_0)=2m\pi,\]
where $m\in\mathbb{Z}$.
Since $z_0\in\Omega$ is arbitrarily chosen, formula \eqref{eq:1118b} implies that 
\[\frac{1}{2\pi}\widetilde u_{\alpha}(z')\in\mathbb{Z}\]
for any $\alpha\in\Pi_1$ and $z'\in\Delta$. $\frac{1}{2\pi}\widetilde u_{\alpha}$ is therefore an integer-valued function on $\Delta$ for any $\alpha\in\Pi_1$.

Now we assume that $\frac{1}{2\pi}\widetilde u_{\alpha}(z')$ are integer-valued functions on $\Delta$ for all $\alpha\in\Pi_1$. 
Recall that 
$\frac{1}{2\pi}\left(\widetilde u_{\alpha}(z')-\int_{\gamma_{\alpha}}\widetilde{d}G_{\Omega}(\cdot,p(z'))\right)\in\mathbb{Z}$
for $\alpha\in\Pi_1$, $\gamma_{\alpha}$ a closed smooth curve which can represent $\alpha$, and $z'\in \Delta\backslash{p^{-1}(\gamma_{\alpha})}$.
Therefore, we have 
\[\frac{1}{2\pi}\int_{\gamma_{\alpha}}\widetilde{d}G_{\Omega}(\cdot,z_0)\in\mathbb{Z}\]
for $\alpha\in\Pi_1$.
By Lemma \ref{lemma1}, we have $\chi_{z_0}=1$, i.e., there exists a holomorphic function $f_{z_0}$ on $\Omega$ s.t. $|f_{z_0}|=e^{G_{\Omega}(\cdot,z_0)}$, 
which implies that $\Omega$ is conformally equivalent to the unit disc less a closed set of inner capacity zero (see \cite{suita}).

\subsection{Proof of Corollary \ref{coro2}.}\quad\par

Let $V_j\subset\Delta\backslash p^{-1}(g^{-1}(0))$ be a simply connected neighborhood of $z_j'$ for $1\le j\le m$ such that $V_i\cap V_j=\emptyset$ for $i\neq j$. Therefore we can assume that $\widetilde{u}_{\alpha_k}$ is the real part of a holomorphic function $f_k$ on $V_j$ for $1\le j\le m$ and $1\le k\le m$. 
Then by Cauchy-Riemann formula,
\[\left(\frac{\partial f_k}{\partial z}(z_j')\right)_{1\le k\le m,1\le j\le m}\] 
is nondegenerate. 

For $1\le k\le m$, we define 
\[F_k(\hat{z}_1,\ldots,\hat{z}_m):=\sum_{j=1}^m f_k(\hat{z}_j),\]
where $(\hat{z}_1,\ldots,\hat{z}_m)\in V_1\times\cdots\times V_m$. Then $F_k$ is a holomorphic function on $V_1\times\cdots\times V_m$. Hence $(F_1,\ldots,F_m)$ define a holomorphic mapping from $V_1\times\cdots\times V_m$ to $\mathbb{C}^m$, and its Jacobian at $(z_1',\ldots,z_m')$ is
\[\left(\frac{\partial F_k}{\partial \hat{z}_j}(z_1',\ldots,z_m')\right)_{1\le k\le m,1\le j\le m}=\left(\frac{\partial f_k}{\partial z}(z_j')\right)_{1\le k\le m,1\le j\le m},\]
which is nondegenerate.

By the inverse mapping theorem, 
\[\left(\sum_{j=1}^m f_1(\hat{z}_j),\ldots,\sum_{j=1}^m f_m(\hat{z}_j)\right)\]
maps a neighborhood of $(z_1',\ldots,z_m')$ conformally onto a domain in $\mathbb{C}^m$,
hence there exists $m$ different points $\{\widetilde{z}'_1,\ldots,\widetilde{z}'_m\}\subset\Delta\backslash p^{-1}(g^{-1}(0))$ in a neighborhood of
$(z'_1,\ldots,z'_m)$ such that 
\[\sum_{j=1}^m\left(\frac{1}{2\pi}\right)\widetilde{u}_{\alpha_k}(\widetilde{z}_j')\in\mathbb{Q}\]
for $1\le k\le m$. We can choose $k_1+1=\cdots=k_m+1$ such that
\[\sum_{j=1}^m\left(\frac{k_j+1}{2\pi}\right)\widetilde{u}_{\alpha_{k}}(\widetilde{z}_j')\in \mathbb{Z}\]
for $1\le k\le m$. Since $u\equiv 0$ and $\frac{1}{2\pi}\widetilde{u}_{\alpha}$ are integer-valued constant functions on the unit disc $\Delta$ for other $\alpha\in\Pi_1$,
\[\sum_{j=1}^m\left(\frac{k_j+1}{2\pi}\right)\widetilde{u}_{\alpha}(\widetilde{z}_j')+\frac{1}{2\pi}\int_{\gamma_{\alpha}}\widetilde{d}u\in \mathbb{Z}\]
for all $\alpha\in\Pi_1$.
Choosing ${z}_j:=p(\widetilde{z}_j')\in\Omega\backslash g^{-1}(0)$ for $1\le j\le m$,  Corollary \ref{coro2}
holds by  Theorem \ref{thm1}.

\section{Proofs of Theorem \ref{thm2}, Remark \ref{remark1}, Corollary \ref{c:n}, Theorem \ref{thm9}, Theorem \ref{thm10} and Theorem \ref{thm11}}
\label{section2}
In this section, we prove Theorem \ref{thm2}, Remark \ref{remark1}, Corollary \ref{c:n}, Theorem \ref{thm9}, Theorem \ref{thm10} and Theorem \ref{thm11}.

\subsection{Proofs of Theorem \ref{thm2} and Remark \ref{remark1}.}\quad\par
Assume that $\Pi_1=\{[\gamma_1],\ldots,[\gamma_{n-1}]\}$. 
Recall that for $1\le k\le n-1$, $\widetilde{u}_k(z')$ is a harmonic function on the unit disc $\Delta$ satisfying the condition that
\begin{align}
\label{align1}
\frac{1}{2\pi}\left(\widetilde u_{k}(z')-\int_{\gamma_{k}}\widetilde{d}G_{\Omega}(\cdot,p(z'))\right)\in\mathbb{Z}
\end{align}
for $z'\in\Delta\backslash p^{-1}(\gamma_k)$, where $p:\Delta\rightarrow\Omega$ is the universal covering. As $\Omega$ is a domain in $\mathbb{C}$, $p_*\widetilde{u}_{k}$ 
is a single-valued harmonic function on $\Omega$ (see the appendix in \cite{G-S-Y}).
We can deduce from condition \ref{align1} that
\begin{align}
\label{align2}
\frac{1}{2\pi}\left( p_*\widetilde{u}_{k}(z)-\int_{\gamma_{k}}\widetilde{d}G_{\Omega}(\cdot,z)\right)\in\mathbb{Z}
\end{align}
for $z\in\Omega\backslash\gamma_k$.

For any $1\le k\le n-1$ and any point $z_0\in\Omega\backslash\gamma_k$, since the boundary of $\Omega$ is analytic, $G_{\Omega}(\cdot,z_0)=0$ on $\partial\Omega$, $G_{\Omega}(\cdot,z_0)$ can be extended to a harmonic function on $U\backslash\{z_0\}$ where $U$ is a 
neighborhood of $\overline{\Omega}$, thanks to the symmetry principle due to Riemann and Schwarz (see \cite{nehari}, p187).
We denote this function also by $G_{\Omega}(\cdot,z_0)$.
Since $\gamma_k$ is a closed smooth curve in $\Omega$ which winds around points in $A_k$ once, and does not wind around points in $A_l$ for $l\neq k$, 
$\gamma_k$ is therefore homologous to $\Gamma_k$ in $U$. Similar arguments as proof of lemma \ref{lemma2} imply that
\[\frac{1}{2\pi}\int_{\Gamma_k}\frac{\partial G_{\Omega}(z,z_0)}{\partial n}ds(z)-\frac{1}{2\pi}\int_{\gamma_k}\frac{\partial G_{\Omega}(z,z_0)}{\partial n}ds(z)\in\mathbb{Z}.\]
It follows from the Green third formula (see Lemma \ref{lemma10}) that  
\[u_k(z_0)=\frac{1}{2\pi}\int_{\Gamma_k}\frac{\partial G(z,z_0)}{\partial n}ds(z)\]
for $1\le k\le n-1$ and $z_0\in\Omega\backslash\gamma_k$, where $u_k\in C(\overline{\Omega})$ satisfies that $u_k=1$ on $\Gamma_k$, $u_k(z)=0$ on $\Gamma_l$ for any $l\neq k$ and $u_k$ is harmonic on $\Omega$. 
As $\Omega$ is a   planar domain, we have $$\frac{\partial v}{\partial n}ds=\widetilde{d}v$$  holds for smooth function $v$ along 
$\gamma_k$ for $1\le k\le n-1$ (see \cite{G-S-Y}).
Therefore, we have 
\begin{align}
\label{align3}
u_k(z)-\frac{1}{2\pi}\int_{\gamma_{k}}\widetilde{d}G_{\Omega}(\cdot,z)\in\mathbb{Z}
\end{align}
for $1\le k\le n-1$ and $z\in\Omega\backslash\gamma_k$. 
Conditions \eqref{align2} and  \eqref{align3} imply that
\[u_k(z)-\frac{1}{2\pi} p_*\widetilde{u}_{k}(z)\in \mathbb{Z}\]
for $1\le k\le n-1$ and $z\in\Omega\backslash\gamma_k$.

Note that $u_k(z)-\frac{1}{2\pi} p_*\widetilde{u}_{k}(z)$ is a harmonic function on $\Omega$, it is therefore a constant integer-valued function.
Remark \ref{remark1} is therefore proved, and Theorem \ref{thm2} can be proved by combining Theorem \ref{thm1} and Remark \ref{remark1}.

\subsection{Proof of Corollary \ref{c:n}}
\quad\par
Let $u_k$ $(1\le k\le n-1)$ be the harmonic function in Theorem \ref{thm2}, which is the solution of a Dirichlet problem. Thus, we have 
\begin{equation}
\label{eq:1120a}0< u_k< 1
\end{equation}
on $\Omega$ and 
\begin{equation}
	\label{eq:1120b}0< \sum_{1\le k\le n-1}u_k< 1
\end{equation}
on $\Omega$. If there exist $m$ different points $z_1,\ldots,z_m$ and  $m$ numbers $a_1,\ldots,a_m$ such that equality \eqref{eq:1120} holds, it follows from Theorem \ref{thm2} that 
$$\sum_{1\le j\le m}(k_j+1)u_{k}(z_j)\in \mathbb{Z}.$$
Following from inequality \eqref{eq:1120a} and \eqref{eq:1120b}, we have $\sum_{1\le j\le m}(k_j+1)u_{k}(z_j)\ge1$, which deduces 
\begin{equation}\nonumber
	\begin{split}
n-1&\le\sum_{1\le k\le n-1}\sum_{1\le j\le m}(k_j+1)u_{k}(z_j)\\
&=\sum_{1\le j\le m}(k_j+1)\sum_{1\le k\le n-1}u_{k}(z_j)\\
&<\sum_{1\le j\le m}(k_j+1).
	\end{split}
	\end{equation}
Then we have $n\le \sum_{1\le j\le m}(k_j+1)$.  Corollary \ref{c:n} holds.

\subsection{Proof of Theorem \ref{thm9}.}\quad\par
Using Lemma \ref{lemma9}, we can choose $m$ different points $z'_1,\ldots,z'_m$ in $\Omega$ s.t. the matrix
\[\left(\frac{\partial f_k}{\partial z}(z_j')\right)_{1\le k\le n-1,1\le j\le m}\]
has rank $n-1$, where $f_k$ is a holomorphic function on $V_0:=\cup_{j=1}^{m} V_j$ s.t.  $u_k$ is 
the real part of $f_k$ and $V_j$ is a simply-connected neighborhood of $z_j'$ s.t. $V_j\cap V_l=\emptyset$ for $j\neq l$.

For $1\le k\le n-1$, we define
\[F_{k}(\hat{z}_1,\ldots,\hat{z}_{m}):=\sum_{j=1}^{m}f_k(\hat{z}_j),\] 
which is a holomorphic function with respect to $(\hat{z}_1,\ldots,\hat{z}_{m})\in V_1\times\cdot\cdot\cdot\times V_{m}$. 
Hence $(F_1,\ldots,F_{n-1})$ defines a holomorphic mapping from $V_1\times\cdot\cdot\cdot\times V_{m}$ to $\mathbb{C}^{n-1}$. 
Note that
\[\left(\frac{\partial F_k}{\partial \hat{z}_j}(z_1',\ldots,z_m')\right)_{1\le k\le n-1,1\le j\le m}=\left(\frac{\partial f_k}{\partial z}(z_j')\right)_{1\le k\le n-1,1\le j\le m}.\]
$(F_1,\ldots,F_{n-1})$ therefore maps a neighborhood of $(z_1',\ldots,z_{m}')$ holomorphically onto an open set in $\mathbb{C}^{n-1}$, hence there exists a point
$(\widetilde{z}_1,\ldots,\widetilde{z}_{m})\in V_1\times\cdots\times V_{m}$ such that the real part of $F_k$ is a rational number at $(\widetilde{z}_1,\ldots,\widetilde{z}_m)$ for $1\le k\le n-1$, i.e., $u_k(\widetilde{z}_1)+\cdots+u_{k}(\widetilde{z}_{m})\in\mathbb{Q}$.

Choosing a positive integer $k_0$ such that $k_0(u_k(\widetilde{z}_1)+\cdots+u_{k}(\widetilde{z}_{m}))\in\mathbb{Z}$  for any $1\le k \le n-1$, and combining  with Theorem \ref{thm2}, Theorem \ref{thm9} is proved.

\subsection{Proof of Theorem \ref{thm10}.}\quad\par

Using Lemma \ref{lemma9}, we can choose $n-2$ different points $z'_1,\ldots,z'_{n-2}$ in $\Omega$ s.t. 
\[\det\left(\left(\frac{\partial f_k}{\partial z}(z_j')\right)_{1\le k\le n-2,1\le j\le n-2}\right)\neq 0,\]
 where $f_k$ is a holomorphic function on $V_0:=\cup_{j=1}^{n-2} V_j$ s.t.  $u_k$ is 
the real part of $f_k$ for $1\le k\le n-1$ and $V_j$ is a simply-connected neighborhood of $z_j'$ s.t. $V_j\cap V_l=\emptyset$ for $j\neq l$.

For $1\le k\le n-1$, we define
\[F_{k}(\hat{z}_1,\ldots,\hat{z}_{n-2}):=\sum_{j=1}^{n-2}f_k(\hat{z}_j),\]
which is a holomorphic function with respect to $(\hat{z}_1,\ldots,\hat{z}_{n-2})\in V_1\times\cdot\cdot\cdot\times V_{n-2}$. Hence $(F_1,\ldots,F_{n-2})$ defines a holomorphic mapping from $V_1\times\cdot\cdot\cdot\times V_{n-2}$ to $\mathbb{C}^{n-2}$. 
Note that
\[\left(\frac{\partial F_k}{\partial \hat{z}_j}(z_1',\ldots,z_{n-2}')\right)_{1\le k\le n-2,1\le j\le n-2}=\left(\frac{\partial f_k}{\partial z}(z_j')\right)_{1\le k\le n-2,1\le j\le n-2}.\]
Then by the inverse mapping theorem, $F:=(F_1,\ldots,F_{n-2})$ maps a neighborhood $X$ of $(z_1',\ldots,z_{n-2}')$ biholomorphically onto an open set in $\mathbb{C}^{n-2}$.

Since $F_{n-1}(\hat{z}_1,\ldots,\hat{z}_{n-2}):=\sum_{j=1}^{n-2}f_{n-1}(\hat{z}_j)$ is a holomorphic function with respect to $(\hat{z}_1,\ldots,\hat{z}_{n-2})\in V_1\times\cdots\times V_{n-2}$, $F_{n-1}\circ F^{-1}$ is therefore a holomorphic function on $F(X)\subset \mathbb{C}^{n-2}$.
As a result, $U_{n-1}\circ F^{-1}$ is a pluriharmonic function on $F(X)\subset \mathbb{C}^{n-2}$, where $U_{n-1}$ is the real part of $F_{n-1}$.

We now prove this theorem by contradiction. Suppose that for any point \newline
$(\widetilde{z}_1,\ldots,\widetilde{z}_{n-2})\in F(X)$ satisfying that the real part $\widetilde{x}_j$ of $\widetilde{z}_j$ $(\widetilde{z}_j=\widetilde{x}_j+i\widetilde{y}_j)$ is rational for $1\le j\le n-2$, $U_{n-1}\circ F^{-1}(\widetilde{z}_1,\ldots,\widetilde{z}_{n-2})$ is not rational. 
Then by the continuity of pluriharmonic functions, $U_{n-1}\circ F^{-1}(\widetilde{z}_1,\ldots,\widetilde{z}_{n-2})$ must be a constant on the set 
\[\{(\widetilde{z}_1,\ldots,\widetilde{z}_{n-2}):(\widetilde{z}_1,\ldots,\widetilde{z}_{n-2})\in F(X) \&\ \widetilde{x}_j=c_j\text{ for $1\le j\le n-2$}\},\] 
where $c_j\in\mathbb{Q}$, which implies that
\[\frac{\partial{U_{n-1}\circ F^{-1}}}{\partial \widetilde{y}_j}(\widetilde{z}_1,\ldots,\widetilde{z}_{n-2})\equiv 0\] 
for $1\le j\le n-2$ on the set 
\[\{(\widetilde{z}_1,\ldots,\widetilde{z}_{n-2}):(\widetilde{z}_1,\ldots,\widetilde{z}_{n-2})\in F(X) \&\ \widetilde{x}_j=c_j \text{ for $1\le j\le n-2$}\},\] 
where $\widetilde{y}_j$ denotes the imaginary part of $\widetilde{z}_j$. Since 
\[\{(\widetilde{z}_1,\ldots,\widetilde{z}_{n-2}):(\widetilde{z}_1,\ldots,\widetilde{z}_{n-2})\in F(X) \&\ \widetilde{x}_j\in\mathbb{Q} \text{ for $1\le j\le n-2$}\}\]
is a dense subset of $F(X)$,
\[\frac{\partial{U_{n-1}\circ F^{-1}}}{\partial \widetilde{y}_j}(\widetilde{z}_1,\ldots,\widetilde{z}_{n-2})\equiv 0\] 
on $F(X)$ for $1\le j\le n-2$.

Since ${U_{n-1}\circ F^{-1}}$ is pluriharmonic with respect to $(\widetilde{z}_1,\ldots,\widetilde{z}_{n-2})$,
\[\frac{\partial{U_{n-1}\circ F^{-1}}}{\partial \widetilde{x}_j^2}(\widetilde{z}_1,\ldots,\widetilde{z}_{n-2})\equiv 0\] 
on $F(X)$ for $1\le j\le n-2$.
As a result, shrinking $X$, we have
\[U_{n-1}\circ F^{-1}(\widetilde{z}_1,\ldots,\widetilde{z}_{n-2})=c_0+\sum_{j=1}^{n-2}c_j\widetilde{x}_j+\sum_{j_1<j_2}c_{j_1j_2}\widetilde{x}_i \widetilde{x}_j+\cdots+c_{12\ldots n-2}\widetilde{x}_1\widetilde{x}_2\cdots\widetilde{x}_{n-2},\] 
where
 $c_{j_1 j_2\ldots j_k}$ ( $1\le j_1<l_2<\cdots<l_k\le n-2$ ) is a constant for $1\le k\le n-2$. 

Recall that $U_{n-1}(\hat z_1,\ldots,\hat z_{n-2})=\sum_{j=1}^{n-2}u_{n-1}(\hat z_j)$. We can therefore conclude that
\begin{align}
\sum_{j=1}^{n-2}u_{n-1}(\hat z_j)&=c_0+\sum_{k=1}^{n-2}\left(c_k\sum_{j=1}^{n-2}u_{k}(\hat z_j)\right)+\sum_{j_1<j_2}\left(c_{j_1j_2}\sum_{j=1}^{n-2}u_{j_1}(\hat z_j)\sum_{j=1}^{n-2}u_{j_2}(\hat z_j)\right)+\cdots\notag\\
&+\left(c_{12\ldots n-2}\sum_{j=1}^{n-2}u_{1}(\hat z_j)\sum_{j=1}^{n-2}u_{2}(\hat z_j)\cdots\sum_{j=1}^{n-2}u_{n-2}(\hat z_j)\right)\notag
\end{align}
holds for $(\hat z_1,\ldots,\hat z_{n-2})\in X\subset V_1\times\cdot\cdot\cdot\times V_{n-2}$, hence for
$(\hat z_1,\ldots,\hat z_{n-2})\in\Omega^{n-2}$ by the identity property of real analytic functions. 

Letting $\hat z_j\rightarrow\Gamma_{n-1}$ for $1\le j\le n-2$, we know that $c_0=n-2$. 
Letting $z_j\rightarrow\Gamma_{n}$ for $1\le j\le n-2$, we can conclude that $0=n-2$, which contradicts to $n\ge3$.
Hence there exists a point
$(z_1'',\ldots,z_{n-2}'')\in V_1\times\cdot\cdot\cdot\times V_{n-2}$ such that the real part of $F_{k}(z_1'',\ldots,z_{n-2}'')$ is a rational number for $1\le k\le n-1$.
Choosing a positive integer $k_0$ such that  $k_0(u_k(z_1'')+\cdots+u_{k}(z_{n-2}''))\in\mathbb{Z}$  for any $1\le k \le n-1$, and combining with Theorem \ref{thm2}, Theorem \ref{thm10} is proved.

\subsection{Proof of Theorem \ref{thm11}.}\quad\par
Let $a\in (0,1)$ be a constant. For $j\in\{0,1,\ldots,m\}$, define
\[\varphi_j(z)=\frac{ae^{\frac{2\pi j\sqrt{-1}}{m+1}}-z}{1-ae^{\frac{2\pi j\sqrt{-1}}{m+1}}z}.\]
It can be checked that each $\varphi_j$ is a holomorphic automorphism of $\Delta$ and 
\[\varphi_j(\varphi_j(z))=z.\]
Denote by $\Delta(r):=\{z\in\mathbb{C}:|z|<r\}$, where $r\in (0,1)$. When $z\in \Delta(r)$, 
\[\varphi_0(z)-a=\frac{a-z}{1-az}-a=\frac{a^2z-z}{1-az},\]
which implies that
\[|\varphi_0(z)-a|\le |a+1|r=(a+1)r.\]
Similarly, when $z\in \Delta(r)$, $j\in\{0,1,\ldots,m\}$,
\[|\varphi_j(z)-ae^{\frac{2\pi j\sqrt{-1}}{m+1}}|\le |ae^{\frac{2\pi j\sqrt{-1}}{m+1}}+1|r\le (a+1)r.\]
We can fix an $r_0\in (0,1)$ so small that when $j\neq k$,
\[\varphi_j(\Delta(r_0))\cap \varphi_k(\Delta(r_0))=\emptyset.\]
It follows from $\varphi_j(\varphi_j(z))=z$ that $|\varphi_j(z)|<r$ if and only if $z\in\varphi_j(\Delta(r))$, where $r\in (0,1)$.
Let $\epsilon\in (0,r_0^{(M+1)m})$. It can be deduced from the above property that
\[\frac{\log |\varphi_j|}{\log \epsilon}\]
is a positive harmonic function on $\Delta\backslash\{ae^{\frac{2\pi j\sqrt{-1}}{m+1}}\}$ such that
\[\frac{\log |\varphi_j|}{\log \epsilon}=1\]
on $\partial\left(\varphi_j(\Delta(\epsilon))\right)$ and
\[\frac{\log |\varphi_j|}{\log \epsilon}<\frac{1}{(M+1)m}\]
on $\Delta\backslash \varphi_j(\Delta(r_0))$.

Let $D_j:=\overline{\varphi_j(\Delta(\epsilon))}$ for $j\in\{0,1,\ldots,m\}$. Let $D_{m+1},\ldots,D_{n-2}$ be arbitrary disjoint closed
disks in $\Delta\backslash\left(\cup_{j=0}^{m} D_j\right)$. Denote by $\Gamma_j:=\partial D_j$ for $j\in\{0,1,\ldots,n-2\}$, and $\Gamma_{n-1}:=\partial\Delta$. Then
\[\Omega:=\Delta\backslash\left(\cup_{j=0}^{n-2} D_j\right)\]
is an $n$-connected domain bounded by $\Gamma_0,\ldots,\Gamma_{n-1}$.

For $j\in\{0,1,\ldots,m\}$, let $u_j$ be the harmonic function on $\Omega$ such that
$u_j|_{\Gamma_j}=1$ and $u_j|_{\Gamma_k}=0$ for $k\neq j$. We know that 
\[\frac{\log |\varphi_j|}{\log \epsilon}\ge u_j\]
on $\partial\Omega$, which implies that
\[\frac{\log |\varphi_j|}{\log \epsilon}\ge u_j\]
on $\Omega$ by the property of harmonic functions. In particularly, 
\[u_j<\frac{1}{(M+1)m}\]
on $\Omega\backslash\varphi_j(\Delta(r_0)\backslash\overline{\Delta(\epsilon)})$.

Let $z_1,\ldots,z_m$ be arbitrary distinct points of $\Omega$. Since 
\[\varphi_j(\Delta(r_0))\cap \varphi_k(\Delta(r_0))=\emptyset\]
for $j\neq k$, there must exist a 
$j_0\in\{0,1,\ldots,m\}$ such that none of $z_1,\ldots,z_m$ lies in $\varphi_{j_0}(\Delta(r_0)\backslash\overline{\Delta(\epsilon)})$.
As a result, 
\[u_{j_0}(z_l)<\frac{1}{(M+1)m}\]
for $l\in\{1,\ldots,m\}$.
Let $k_1,\ldots,k_m$ be arbitrary nonnegative integers such that $k_l\le M$ for each $l$, then 
\[\sum_{l=1}^m (k_l+1)u_{j_0}(z_l)\in (0,1),\]
which is not an integer.
Theorem \ref{thm11} is therefore proved by combining the above result with Theorem \ref{thm2}.

\section{Appendix}
In this section, we assume that  $\Omega$ is a product of open Riemann surfaces, and we consider the corresponding weighted version of Suita conjecture for higher derivatives and finite points case.

Let $\Omega_j$ be an open Riemann surface for $1\le j\le n$, which admits a nontrivial Green function $G_{\Omega_j}$. Let $\Omega=\prod_{1\le j\le n}\Omega_j$ be an $n$-dimensional complex manifold, and let $\pi_j$
be the natural projection from $\Omega$ to $\Omega_j$.

Let $Z_j=\{z_{j,1},z_{j,2},\ldots,z_{j,m_j}\}\subset \Omega_j$ for any $j\in \{1,2,\ldots,n\}$, where $m_j$ is a positive integer. Let $w_{j,k}$ be a local coordinate on a neighborhood $V_{z_{j,k}}\Subset\Omega_j$ satisfying $w_{j,k}(z_{j,k})=0$
for any $j\in\{1,2,\ldots,n\}$ and $k\in\{1,2,\ldots,m_j\}$, where $V_{z_{j,k}}\cap V_{z_{j,k'}}=\emptyset$ for any $j$ and $k\neq k'$. Denote that $I_1:=\{(\beta_1,\beta_2,\ldots,\beta_n):1\le\beta_j\le m_j\}$ for any $j\in\{1,2,\ldots,n\}$, 
$V_{\beta}:=\prod_{1\le j\le n}V_{z_{j,\beta_j}}$ for $\beta=(\beta_1,\beta_2,\ldots,\beta_n)\in I_1$ and $w_{\beta}:=(w_{1,\beta_1},w_{2,\beta_2},\ldots,w_{n,\beta_n})$ is a local coordinate on $V_{\beta}$ of $z_{\beta}:=(z_{1,\beta_1},z_{2,\beta_2},\ldots,z_{n,\beta_n})\in\Omega$.

Let $\varphi_j$ be a subharmonic function on $\Omega_j$, such that $\varphi:=\sum_{1\le j\le n}\pi_j^*(\varphi_j)$ satisfy that $\varphi (z_{\beta})>-\infty$ for any $\beta\in I_1$. Denote by
\[\psi:= \max_{1\le j\le n} \left\{2\sum_{1\le k\le m_j}p_{j,k}\pi_j^*(G_{\Omega_j}(\cdot,z_{j,k}))\right\},\]
where $p_{j,k}$ is a positive real number.
Let $f$ be a holomorphic $(n,0)$ form on $\cup_{\beta\in I_1}V_{\beta}$, such that
$f=\sum_{\alpha\in E_{\beta}}d_{\beta,\alpha}w_{\beta}^{\alpha} dw_{1,\beta_1}\wedge dw_{2,\beta_2}\wedge\ldots\wedge dw_{n,\beta_n}$ on $V_{\beta}$ for any $\beta\in I_1$, where 
$E_{\beta}:=\{(\alpha_1,\alpha_2,\ldots,\alpha_n):\sum_{1\le j\le n}\frac{\alpha_j+1}{p_{j,\beta_j}}=1\& \alpha_j\in\mathbb{Z}_{\ge 0} \}$. Assume that $f=w_{\beta^*}^{\alpha_{\beta^*}} dw_{1,1}\wedge dw_{2,1}\wedge\ldots\wedge dw_{n,1}$ on $V_{\beta^*}$, where $\beta^*=(1,1,\ldots,1)$. Denote by 
\[ c_{j,k}:=\exp\lim_{z\rightarrow z_{j,k}}\left(\frac{\sum_{1\le k_1\le m_j}p_{j,k_1}G_{\Omega_j}(z,z_{j,k_1})}{p_{j,k}}-\log|w_{j,k}(z)|  \right) \]
for any $j\in\{1,2,...,n\}$ and $k\in\{1,2,...,m_j\}$.
 Let $c$ be a positive function on $(0,+\infty)$, such that $\int_0^{+\infty}c(t)e^{-t}ds<+\infty$ and $c(t)e^{-t}$ is decreasing on $(0,+\infty)$.

In \cite{G-Y3}, Guan-Yuan obtained the following result.
\begin{Theorem}
	\label{thm5}
There exists a holomorphic $(n,0)$ form F on $\Omega$ satisfying that
\[  (F-f,z_{\beta}) \in (\mathcal{O}(K_M)\otimes \mathcal{I}(\psi) )_{z_{\beta}} \]
for any $\beta\in I_1$ and
\[\int_{\Omega}|F|^2e^{-\varphi}c(-\psi)\le\left(\int_0^{+\infty}c(s)e^{-s}ds\right)\sum_{\beta\in I_1}\sum_{\alpha\in E_{\beta}}\frac{|d_{\beta,\alpha}|^2(2\pi)^n e^{-\varphi(z_{\beta})}}{\prod_{1\le j\le n}(\alpha_j+1)c_{j,\beta_j}^{2\alpha_j+2}}.\]

Moreover,  denoting $C_{f,\psi}:=\inf\{\int_{\Omega}|\widetilde{F}|^2e^{-\varphi}c(-\psi):\widetilde{F}\in H^0(\Omega,\mathcal{O}(K_M))\& (\widetilde{F}-f,z_{\beta})\in (\mathcal{O}(K_M)\otimes \mathcal{I}(\psi) )_{z_{\beta}}$ for any $\beta\in I_1\}$, equality $$C_{f,\psi}=\left(\int_0^{+\infty}c(s)e^{-s}ds \right)\sum_{\beta\in I_1}\sum_{\alpha\in E_{\beta}}\frac{|d_{\beta,\alpha}|^2(2\pi)^n e^{-\varphi(z_{\beta})}}{\prod_{1\le j\le n}(\alpha_j+1)c_{j,\beta_j}^{2\alpha_j+2}}$$
holds if and only if the following statements hold:

$(1)$ $\varphi_j=2\log|g_j|+2u_j$ for any $j\in\{1,2,\ldots,n\}$, where $u_j$ is a harmonic function on $\Omega_j$ and $g_j\in\mathcal{O}(\Omega_j)$ satisfying $g_j(z_{j,k})\neq 0$ for any $k\in\{1,2,...,m_j\}$; 

$(2)$ there exists a nonnegative integer $\gamma_{j,k}$ for any $j\in\{1,2,...,n\}$ and $k\in\{1,2,...,m_j\}$, which satisfies that $\prod_{1\le k\le m_j}\chi_{j,z_{j,k}}^{\gamma_{j,k}+1}=\chi_{j,-u_j}$ and $\sum_{1\le j\le n}\frac{\gamma_{j,\beta_j+1}}{p_{j,\beta_j}}=1$ for any $\beta\in I_1$, where $\chi_{j,z_j}$ and $\chi_{j,-u_j}$ is the characters
of $\Omega_j$ associated to the function $G_{\Omega_j}(\cdot,z_j)$ and $-u_j$ respectively;

$(3)$ $f=(c_{\beta}\prod_{1\le j\le n}w_{j,\beta_j}^{\gamma_{j,\beta_j}}+g_{\beta})dw_{1,\beta_1}\wedge dw_{2,\beta_2}\wedge\ldots\wedge dw_{n,\beta_n}$ on $V_{\beta}$ for any $\beta\in I_1$, where $c_{\beta}$ is a constant and $g_{\beta}$ is a holomorphic function on $V_{\beta}$, such that $(g_{\beta},z_{\beta})\in \mathcal{I}(\psi)_{z_{\beta}}$;

$(4)$ $\lim_{z\rightarrow z_{\beta}}\frac{c_{\beta}\prod_{1\le j\le n}w_{j,\beta_j}^{\gamma_{j,\beta_j}} dw_{1,\beta_1}\wedge dw_{2,\beta_2}\wedge\ldots\wedge dw_{n,\beta_n} }{\bigwedge_{1\le j\le n}\pi_j^*\left(g_j(P_j)_*\left(f_{u_j}\left(\prod_{1\le k\le m_j}f_{z_j,l}^{\gamma_{j,k}+1}\right)\left(\sum_{1\le k\le m_j}p_{j,k}\frac{df_{z_{j,k}}}{f_{z_j,k}}\right)\right)\right)}=c_0$
for any $\beta\in I_1$, where $P_j:\Delta\rightarrow\Omega_j$ is the universal covering map, $c_0\in\mathbb{C}\backslash\{0\}$ is a constant independent of $\beta$, and $f_{u_j}$ is a holomorphic function on $\Delta$, such that $|f_{u_j}|=P_j^*(e^{u_j})$ and $f_{z_j,k}$ is a holomorphic function on $\Delta$ such that 
$|f_{u_j}|=P_j^*(e^{G_{\Omega_j}(\cdot,z_{j,k})})$ for any $j\in\{1,2,\ldots,n\}$ and $k\in\{1,2,\ldots,m_j\}$.

\end{Theorem}

For each $1\le j\le n$, let $\Pi_{1,j}$ be a subset of the fundamental group $\pi_1(\Omega_j)$ which generates $\pi_1(\Omega_j)$ (when $\Omega_j$ is simply connected, set $\Pi_{1,j}=\{0\}$). 
For every $\alpha\in \Pi_{1,j}$, $\gamma_{\alpha}$ denotes a piecewise smooth closed curve on $\Omega_j$, which represents $\alpha$.
Let $P_j:\Delta\rightarrow\Omega_j$ be the universal covering of $\Omega_j$.
By Lemma \ref{lemma7}, there exists a harmonic function $\widetilde u_{\alpha}$ on $\Delta$ such that
\begin{equation}\label{eq:9}
	\frac{1}{2\pi}\left(\widetilde u_{\alpha}(z')-\int_{\gamma_{\alpha}}\widetilde{d}G_{\Omega}(\cdot,P_j(z'))\right)\in\mathbb{Z}\notag
\end{equation}
for any $z'\in \Delta\backslash{P_j^{-1}(\gamma_{\alpha})}$. 

Now we assume that the following statements holds:

$(1)$ $\varphi_j=2\log|g_j|+2u_j$ for any $j\in\{1,2,\ldots,n\}$, where $u_j$ is a harmonic function on $\Omega_j$ and $g_j\in\mathcal{O}(\Omega_j)$ satisfying $g_j(z_{j,k})\neq 0$ for any $k\in\{1,2,\ldots,m_j\}$; 

$(2)$ there exists a nonnegative integer $\gamma_{j,k}$ for any $j\in\{1,2,\ldots,n\}$ and $k\in\{1,2,\ldots,m_j\}$, which satisfies that $\sum_{1\le j\le n}\frac{\gamma_{j,\beta_j+1}}{p_{j,\beta_j}}=1$ for any $\beta\in I_1$. Let $f=(c_{\beta}\prod_{1\le j\le n}w_{j,\beta_j}^{\gamma_{j,\beta_j}})dw_{1,\beta_1}\wedge dw_{2,\beta_2}\wedge\ldots\wedge dw_{n,\beta_n}$ on $V_{\beta}$ for any $\beta\in I_1$, where $c_\beta$ is a constant.

Then we have the following result.
\begin{Theorem}
\label{thm4}
There exists $\{c_\beta\}_{\beta\in I_1}$ such that equality $$C_{f,\psi}=\left(\int_0^{+\infty}c(s)e^{-s}ds \right)\sum_{\beta\in I_1}\sum_{\alpha\in E_{\beta}}\frac{|d_{\beta,\alpha}|^2(2\pi)^n e^{-\varphi(z_{\beta})}}{\prod_{1\le j\le n}(\alpha_j+1)c_{j,\beta_j}^{2\alpha_j+2}}$$ 
holds if and only if for each $1\le j\le n$,
\begin{equation}
\label{eq8}\sum_{k=1}^{m_j}\left(\frac{\gamma_{j,k}+1}{2\pi}\right)\widetilde{u}_{\alpha}(\widetilde{z_{j,k}})+\frac{1}{2\pi}\int_{\gamma_{\alpha}}\widetilde{d}u_j\in \mathbb{Z} 
\end{equation}
holds for any $\alpha\in \Pi_{1,j}$ and $\widetilde{z_{j,k}}\in p_j^{-1}(z_{j,k})$ for $1\le k\le m_j$. 
\end{Theorem}

\begin{proof}
	By lemma \ref{lemma2}, we know that for each $1\le j\le n$,
	\[\chi_{-u_j}=\prod_{1\le k\le m_j}\chi_{z_{j,k}}^{\gamma_{j,k}+1}\]
	if and only if
	\begin{align}
	\label{align5}
	e^{-i\int_{\gamma_{\alpha}}\widetilde{d}u_j}=e^{i\sum_{1\le k\le m_j}(\gamma_{j,k}+1)\int_{\gamma_{\alpha}}\widetilde{d}G_{\Omega}(\cdot,z_{j,k})}
	\end{align}
	holds for any $\alpha\in\Pi_{j,1}$ and piecewise smooth curves $\gamma_{\alpha}\subset\Omega_j\backslash\{z_{j,1},\ldots,z_{j,m_j}\}$ representing $\alpha$.
	By condition \eqref{3.1}, 
	we know that for each $1\le j\le n$, equality \eqref{align5} holds if and only if
	\[\sum_{k=1}^{m_j}\left(\frac{\gamma_{j,k}+1}{2\pi}\right)\widetilde{u}_{\alpha}(\widetilde{z_{j,k}})+\frac{1}{2\pi}\int_{\gamma_{\alpha}}\widetilde{d}u_j\in \mathbb{Z}\]
	holds for any $\alpha\in \Pi_{1,j}$ and $\widetilde{z_{j,k}}\in p_j^{-1}(z_{j,k})$ for $1\le k\le m_j$. 
	Combining this with Theorem \ref{thm5}, we know that
	Theorem \ref{thm4}  holds.
\end{proof}

\begin{Remark}
Similarly to Theorem \ref{thm4}, Theorem \ref{thm2}, \ref{thm9} and \ref{thm10}
can also be generalized to the corresponding product version on $\Omega=\prod_{1\le j\le n}\Omega_j$.
\end{Remark}


\vspace{.1in} {\em Acknowledgements}. The authors would like to thank  Dr. Shijie Bao and Dr. Zhitong Mi for checking the manuscript. 
The first named author was supported by National Key R\&D Program of China 2021YFA1003100 and NSFC-12425101. The third author was supported by China Postdoctoral Science Foundation BX20230402 and 2023M743719.

\bibliographystyle{references}
\bibliography{xbib}

\begin{thebibliography}{100}
	
\bibitem{Ahlfors74}L.V. Ahlfors and L. Sario, Riemann Surfaces, Princeton University Press, Princeton, NJ, 1974.



\bibitem{Blocki}Z. Blocki, Suita conjecture and the Ohsawa-Takegoshi extension theorem, Invent. Math. 193(2013), 149-158.

\bibitem{demailly1}J.-P. Demailly, Complex analytic and differential geometry, electronically accesssiable at https://www-fourier.ujf-grenoble.fr/\textasciitilde demailly/manuscripts/agbook.pdf.





\bibitem{fulton}W. Fulton, Algebraic topology a first course, Graduate texts in Mathematics, Springer-Verlag New York. Inc, 1995.


\bibitem{OF81}O. Forster, Lectures on Riemann surfaces, Grad. Texts in Math., 81, Springer-Verlag, New York-Berlin, 1981.

\bibitem{guan_general concave}Q.A. Guan, General concavity of minimal $L^2$ integrals related to multiplier sheaves, arXiv:1811.03261v4.

\bibitem{guan-mi-peking}Q.A. Guan and Z.T. Mi, Concavity of minimal $L^2$ integrals related to multiplier ideal sheaves, Peking Math. J. 6 (2023), no. 2, 393-457.

\bibitem{G-M-Y}Q.A. Guan, Z.T. Mi and Z. Yuan, Concavity property of minimal $L^2$ integrals with Lebesgue measurable gain $\uppercase\expandafter{\romannumeral2}$,  Adv. Math. 450 (2024), Paper No. 109766, 61 pp., see also http://www.researchgate.net/publication/354464147.

\bibitem{G-S-Y}Q.A. Guan, X. Sun and Z. Yuan, A remark on a weighted version of Suita conjecture for higher derivatives. Math. Z. 307 (2024), no. 1, Paper No. 17, 23 pp.

\bibitem{GY-concavity}Q.A. Guan and Z. Yuan, Concavity property of minimal $L^2$ integrals with Lebesgue measurable gain, Nagoya Math. J. 252 (2023), 842-905.

\bibitem{G-Y2}Q.A. Guan and Z. Yuan, Concavity property of minimal $L^2$ integrals with Lebesgue measurable gain $\uppercase\expandafter{\romannumeral3}$: open Riemann surfaces, arXiv.2211.04951v3.


\bibitem{G-Y3}Q.A. Guan and Z. Yuan, Concavity property of minimal $L^2$ integrals with Lebesgue measurable gain $\uppercase\expandafter{\romannumeral4}$: product of open Riemann surfaces, Peking Math. J. 7 (2024), no. 1, 91-154.





\bibitem{G-ZhouL2_CR}Q.A. Guan and X.Y. Zhou, Optimal constant problem in the $L^2$ extension theorem. C. R. Math. Acad. Sci. Paris 350 (2012), no. 15-16, 753-756.

\bibitem{guan-zhou13ap}Q.A. Guan and X.Y. Zhou, A solution of an $L^{2}$ extension problem with an optimal estimate and applications,
Ann. of Math. (2) 181 (2015), no. 3, 1139-1208.


\bibitem{nehari}Z. Nehari, Conformal mapping, Dover Publication, Inc. New York, 1975.

\bibitem{OhsawaObservation}T. Ohsawa,
Addendum to ``On the Bergman kernel of hyperconvex domain'',
Nagoya
Math. J. 137 (1995), 145-148


\bibitem{L-K}L. Sario and K. Oikawa, Capacity functions, Grundl. Math. Wissen. 149, Springer-Verlag, New York, 1969. 

\bibitem{suita} N. Suita, Capacities and kernels on Riemann surfaces, Arch. Rational Mech. Anal. 46 (1972),
212-217.

\bibitem{X-Z}W. Xu and X.Y. Zhou, Optimal $L^2$ Extensions of Openness Type, arXiv:2202.04791.

\bibitem{yamada}A. Yamada, Topics related to reproducing kernels, theta functions and the Suita conjecture (Japanese), The theory of reproducing kernels and their applications (Kyoto 1998). Sūrikaisekikenkyūsho Kōkyūroku No. 1067 (1998), 39-47.


\end{thebibliography}

\end{document}